\documentclass{conm-p-l}
\usepackage{amssymb, amsmath, amsthm, amscd, amsfonts, dsfont, tikz-cd}
\usepackage{hyperref} 

\newtheorem*{theorem*}{Theorem}
\newtheorem{theorem}{Theorem}[section]
\newtheorem{lemma}[theorem]{Lemma}
\newtheorem{proposition}{Proposition}[section]
\newtheorem{corollary}[proposition]{Corollary}

\theoremstyle{definition}
\newtheorem{definition}[theorem]{Definition}
\newtheorem{example}[theorem]{Example}

\theoremstyle{remark}
\newtheorem{remark}[theorem]{Remark}
\newtheorem{remarks}[proposition]{Remarks}

\newcommand{\Hom}{{\rm Hom}}

\newcommand{\Ext}{{\rm Ext}}

\newcommand{\im}{{\rm Im}\,}

\newcommand{\Id}{{\rm Id}\,}

\numberwithin{equation}{section}

\title{The structure of simply colored coalgebras}
\author{Yang Mo}
\address{Department of Mathematics, Purdue University, West Lafayette, IN 47907}

\begin{document}
	\maketitle

	\begin{abstract} 
	    A simply colored coalgebra is a coassociative counital coalgebra $C$ over an arbitrary ring $R$, which can be decomposed into a direct sum of two $R$-modules: one generated by set-like elements and another consisting of conilpotent elements. Our main result is the equivalence between simply colored coalgebras over a field $k$ and pointed coalgebras with a choice of splitting of its coradical. Additionally, we also prove that the category of simply colored coalgebras is both complete and cocomplete.
	\end{abstract}
	
	\section*{Introduction}
         The aim of this paper is to delve into the subtleties in \cite{kaufmann_mo_2022} and to extend its results. To begin, we briefly discuss the initial motivation for this collaborative work with R. Kaufmann.

        In \cite{MayMoreConcise}, Peter May constructs an component coalgebra in the following sense: Let $C_{g} = Rg \oplus \overline{C}_g$ and $\overline{C}_g = \{x | \Delta(x) = x \otimes g + x' \otimes x'' + g \otimes x, \deg(x') > 0 \text{ and } \deg(x'') > 0\}$ and $g$ is a set-like  \footnote{In this paper, we adopt the unconventional term \textbf{``set-like"} instead of the commonly used term \textbf{``group-like"} found in \cite{MayMoreConcise} \cite{Susan}, \cite{DR2}, \cite{Sweedler}, in order to maintain consistency.} element in $C_g$,. In essence, $C_{g}$ is a connected graded coalgebra over $R$ that lacks primitive elements and is indexed by the unique set-like element in $g$. Given such a family $\{C_g\}$ of them, we define the component coalgebra as the direct sum of coalgebras over the family, i.e. 
        $$C = \oplus_{g} C_{g}$$
        This construction is motivated from the study of loop spaces. It should be noted that each $g$ is also a set-like element in the direct sum $C$.
        A component bialgebra refers to a bialgebra in which its coalgebra is a component coalgebra. If the bialgebra is group-like, meaning every set-like element is invertible under the multiplication, then the bialgebra becomes a Hopf algebra.

        At a certain point in our study of Hopf algebras obtained from Feynman categories, we realized the need for a more general formulation. In a broader sense, the components should be indexed by two potentially distinct set-like elements, denoted as $g$ and $h$, within $C$. In other words, we have
        $$C = \oplus_{g, h}C_{g, h}$$

The good news is that pointed coalgebras (not necessarily graded) fulfill the required property, as shown in \cite{DR2}.
Although each $C_{g,h}$ may not be a coaugmented coalgebra, the existence of an antipode for a pointed coalgebra $C$ remains the same: every set-like element is invertible. However, the bad news is that the notion of pointed coalgebras is not well-defined when the ground ring is not a field. This difficulty arises from the fact that the notion of subcoalgebras may not be well-behaved.

To address this challenge, we introduce the concept of a simply colored coalgebra. It is defined as a triple $(C, G, \delta)$, where $C$ is a coalgebra, $G$ is a set of set-like elements in $C$ (this will be called colors), and $\delta: C \to C[G]$ is a coalgebra retraction that maps $C$ onto the set-like coalgebra $C[G]$. This decomposition allows us to express $C$ as the direct sum $C = C[G] \oplus \ker{\delta}$.

The simply colored coalgebra can be defined over any commutative rings with unity, such as $\mathbb{Z}$. Furthermore, we demonstrate that any non-negatively graded coalgebra with its zero degree freely spanned by set-like elements can be regarded as a simply colored coalgebra. In particular, a connected graded coalgebra is simply colored.

This framework enables us to generalize the results presented in \cite{IRA1} and \cite{IRA2}. For further details, We refer readers to \cite{kaufmann_mo_2022}, which uses different terminology and introduces additional variants.

In this paper, we present the following results to expand upon the topics discussed in \cite{kaufmann_mo_2022}. To be more precise, 

\begin{enumerate}
    \item The reduced comultiplication can be defined more simply in terms of a retraction (Theorem \ref{1}). This allows us to extend \cite[Lemma 2.19]{kaufmann_mo_2022} to any comonoid in an monoidal abelian category. This demonstrates that the use of set-like coalgebras is not necessary in the argument and that the coassociativity of the reduced comultiplication can be proved without it. 
    \item A pointed coalgebra with a splitting is a pointed coalgebra $C$ equipped with an additional coalgebra homomorphism $\delta$ onto its coradical, such that $C$ can be decomposed as $C = C_0 \oplus \ker \delta$. We show that a simply colored coalgebra over a field is equivalent to a pointed coalgebra with a splitting.
    More precisely,
 \begin{theorem*}[\ref{11}] {\ }
     Let $k$ be a field, 
     \begin{enumerate}
         \item Every simply colored coalgebra $(C, G, \delta)$ over $k$ can be viewed as a pointed coalgebra with a splitting $(C, \delta)$.
         \item Conversely, if $(C, \delta)$ is a pointed coalgebra with a splitting, it naturally gives rise to a simply colored coalgebra structure $(C, G, \delta)$, where $G$ is the set of all set-like elements in $C$.
     \end{enumerate}
 \end{theorem*}  

In particular, it is worth noting that coaugmented conilpotent coalgebras can be precisely characterized as pointed coalgebras with a splitting where $G$ consists of a single element. While this fact is likely known in the literature (see, for example, \cite{ckd}), the detailed argument supporting this characterization is not commonly found.
    
    \item In the last section, we show the category of pointed coalgebras with a splitting is both complete and cocomplete. It is worth mentioning that this result is independently proven in \cite{ckd}. However, we will present a different approach that does not rely on pseudocompact algebras.
\end{enumerate}

We should highlight that the characterization of reduced comultiplication (as discussed in Theorem \ref{1} and Section \ref{33}) enables us to define a "pointed-coalgebra-like" cooperad. It provides a generalization of the existing notion of "conilpotent coaugmented-like" conilpotent coaugmented cooperads. As a consequence, it could potentially pave the way for a broader generalization of Koszul duality of operads, similar to the approach taken in \cite{ckd}, which may serve as a subject for future research.

\subsection*{Conventions and notations}
We assume that readers are familiar with the definitions and basic examples of coalgebras, bialgebras, and Hopf algebras. Throughout this paper, we work over a unital commutative associative ring $R$ or a field $k$. Given the commutativity of $R$, an $R$-module naturally becomes a bimodule. The tensor product $\otimes$ is considered to be over the ground ring $R$, unless otherwise stated. We use $\rho_l$ (and $\rho_r$) to denote a left coaction (and right coaction). The symbols $\omega_l$ and $\omega_r$ are reserved for coactions induced via retraction, which we will discuss later in the paper. It $M = N \oplus N'$ can be decomposed, then the notation $\pi_N: M \to N$ is used to represent the canonical projection onto the submodule $N$. If the submodule is clear from the context, we omit the subscript $B$ from $\pi_N$. Finally, in an unconventional manner, we will use the term \textbf{set-like} elements to refer to what is commonly known as group-like elements in most literature.
	\subsection*{Acknowledgement}
 
The author would like to express gratitude to R. Kaufmann, M. Rivera, and L. Guo for their valuable suggestions and insightful comments, which greatly contributed to the improvement of the original ideas. I would like to extend my special thanks to the organizers of the AMS special session on higher structures in topology, geometry, and physics, where I had the opportunity to present a preliminary version of this work.
	
	\section{Preliminaries}
	Most of the material in this section can be found in textbooks such as \cite{DR2}, \cite{Susan}, and \cite{Sweedler}. First, we will provide a brief overview of the fundamental concepts
	
	\begin{definition} {\ }
		\begin{enumerate}
			\item A \textit{cosimple} coalgebra is a coalgebra whose subcoalgebras are exactly zero and itself.
			\item A \textit{coradical} of coalgebra is a sum of all its cosimple subcoalgbras.  
			\item A \textit{cosemisimple} coalgebra is a coalgebra which equals to its own coradical. 
		\end{enumerate}
	\end{definition}
 
According to the fundamental theorem of coalgebras, any cosimple coalgebra must be finite-dimensional. Its dual algebra corresponds precisely to a simple algebra (see \cite{Susan}, Chapter 1). While it is commonly referred to as a simple coalgebra in most literature, we will retain the "co-" prefix for consistency. Additionally, it can be proven that the sum of cosimple coalgebras must be a direct sum (see for example \cite{Susan}).
	
	\begin{example} \cite{DR2}
		Here is an cosimple coalgebra over a field. We consider the (standard) matrix coalgebra $M^c_n(k)$ over a field $k$. As a vector space, it is spanned by elements $e_{ij}$ endowed with the following coalgebra structure
		$$\Delta(e_{ij}) = \sum_{1 \leq t \leq n}e_{it} \otimes e_{tj}$$
		$$\epsilon(e_{ij}) = \delta_{ij}$$
		$M^c_n(k)$ as a vector space is finite dimensional, so its linear dual $(M^c_n(k))^*$ is a algebra over $k$. In fact, one can check it is isomorphic to the standard matrix algebra $M_n(k)$ as $$e^*_{ij}e^*_{kl} = \delta_{jk}e^*_{il}$$ Since the matrix algebra $M_n(k)$ is simple (the only ideal is 0 and itself), by duality the only subcoalgebra is 0 and itself. By definition, $M^c_n(k)$ is a simple coalgebra. $M^c_n(k)$ is pointed if and only if $n = 1$, since the dimension of simple pointed coalgebra can only have dimension of 1 and by comparison the dimension of the matrix coalgebra $M^c_n(k)$ is $n^2$ ($n^2 > 1$ for $n > 1$).
	\end{example}
	\begin{definition}
		Let $C$ be a coalgebra over a field $k$ and $X, Y$ subspaces. Define the \textit{wedge product} $X \wedge Y$  of $X$ and $Y$ to be the kernel of the composition
		\[\begin{tikzcd}
			C && {C \otimes C} && {C/X \otimes C/Y}
			\arrow["{\pi \otimes \pi}", from=1-3, to=1-5]
			\arrow["\Delta", from=1-1, to=1-3]
		\end{tikzcd}\]
	\end{definition}
	
	There are two important observations from the definition, which were left as exercises in \cite{Sweedler}.
	\begin{proposition} \cite{Sweedler}
		\begin{enumerate}
			\item  $(X \wedge Y) \wedge Z = X \wedge (Y \wedge Z)$
			\item $X \wedge Y = \Delta^{-1}(C \otimes Y + X \otimes C)$
		\end{enumerate}
	\end{proposition}
	The first assertion follows from coassociativity $(\Delta \otimes id) \circ \Delta = (id \otimes \Delta) \circ \Delta$ that the kernel of 
	\[\begin{tikzcd}
		C && {C \otimes C \otimes C} && {C/X \otimes C/Y \otimes C/Z}
		\arrow["{\pi \otimes \pi \otimes \pi}", from=1-3, to=1-5]
		\arrow["(\Delta \otimes id) \circ \Delta ", from=1-1, to=1-3]
	\end{tikzcd}\]
	
	is the same as the kernel of 
	\[\begin{tikzcd}
		C && {C \otimes C \otimes C} && {C/X \otimes C/Y \otimes C/Z}
		\arrow["{\pi \otimes \pi \otimes \pi}", from=1-3, to=1-5]
		\arrow["(id \otimes \Delta) \circ \Delta ", from=1-1, to=1-3]
	\end{tikzcd}\]
	As a result, it make senses to use the notation $\wedge^nX = (\wedge^{n-1}X) \wedge X$ without attention to the order. The second assertion follows form the following linear algebra exercise.
	
	\begin{lemma}
		If $f:V_1 \to W_1$ and $g: V_2 \to W_2$ are linear maps between vector spaces. Then 
		$$\ker(f \otimes g) = \ker(f) \otimes V_2 + V_1 \otimes \ker(g)$$
	\end{lemma}
	\begin{proof}
		Every monomorphism of vector spaces splits. Thus the image of $f \otimes g$ is $f(V_1) \otimes f(V_2) \subset W_1 \otimes W_2$. Without loss of generality, assume both $f$ and $g$ surjective. In this case, since the short exact sequence splits, $f$ induces a splitting  as $V_1 = \ker f \oplus V'_1$ s.t. $V'_1 \cong W_1$ via $f$. Similarly, $g$ induces a splitting  as $V_2 = \ker g \oplus V'_2$ s.t. $V'_2 \cong W_2$ via $f$ Then 
		$$V_1 \otimes V_2 = \ker f \otimes \ker g \oplus \ker f \otimes V'_2 \oplus V'_1 \otimes \ker g \oplus V'_1 \otimes V'_2$$
		Note if $s \circ f = id$ and $t \circ g = id$, then $(s \otimes t) \circ (f \otimes g) = (s \circ f) \otimes (t \circ g) = id \otimes id = id$. This implies given the splitting we chose for $f$ and $g$, $im (f \otimes g) = V'_1 \otimes V'_2$ and
		$$\ker(f \otimes g) = \ker f \otimes \ker g \oplus \ker f \otimes V'_2 \oplus V'_1 \otimes \ker g  = \ker(f) \otimes V_2 + V_1 \otimes \ker(g)$$
	\end{proof}
	Let $C_0$ be the coradical for an coalgebra $C$. It is known (pp.125 \cite{DR2}) the sequence of subspaces $C_n = \wedge^nC_0$ forms a exhaustive coalgebra filtration; it is called \textit{coradical filtration} of $C$. 
	\begin{definition}
		A \textit{pointed} coalgebra over a field is defined to be a coalgebra in which every cosimple subcoalgebras is one dimensional.
	\end{definition}
	
Alternatively, a pointed coalgebra can be defined as a coalgebra whose coradical is generated freely by set-like elements, since every one-dimensional coalgebra is necessarily generated by a set-like element. It is essential to note that the set of set-like elements is linearly independent (refer to Lemma 2.1.12 in \cite{DR2} for further details).

\begin{example} {\ }
\begin{enumerate}
\item The simplest form of a pointed coalgebra is given by $kg$, where $g$ is a set-like element. By duality, every one-dimensional cosimple coalgebra is isomorphic to such a coalgebra $kg$. The element $g$ in this case is referred to as a \textit{set-like} element. 
Furthermore, the sum of one-dimensional cosimple coalgebras is also pointed.
\item According to \cite{Sweedler}, a cocommutative coalgebra over an algebraically closed field is pointed.
\item As stated in \cite{DR2}, a graded coalgebra with its degree zero component spanned by group-like elements is pointed. This follows from the observation that the coradical must be contained within the degree one component (Proposition 4.1.2 in \cite{DR2}).
\item Every path coalgebra associated with a quiver is a pointed coalgebra. This can be verified either directly or by utilizing property (2) and noting that the path coalgebra is graded, with its zero component being spanned by set-like elements (vertices). The path coalgebra is coradically graded. Further information can be found in \cite{Chin2003ABI}.
\end{enumerate}
\end{example}

We know that the path coalgebra associated with a quiver is pointed. However, the converse remains an open problem. Nonetheless, we can establish the following weaker converse:

\begin{proposition} \label{15} \cite{d_1997}
Let $C$ be a pointed coalgebra. Then there exists a quiver such that $C$ is isomorphic to an admissible subcoalgebra of the path coalgebra of the quiver.
\end{proposition}

It has been shown in \cite{Susan} that every pointed coalgebra admits a splitting of its coradical. However, the choice of splitting is not unique. To record the information of the splitting $\delta$, we come up with the following notion. 

\begin{definition}
A \textit{pointed coalgebra with a splitting} is a pair $(C, \delta)$ consisting of a pointed coalgebra $C$ and a surjective coalgebra homomorphism $\delta$ from $C$ onto its coradical $C_0$.
\end{definition}

\begin{remark}
    The splitting of coradical is a part of structure, and it is not property of pointed coalgebras. 
\end{remark}

Note such a $\delta$ induces a splitting $C = C_0 \oplus I$ where $I = \ker \delta$. Conversely, if we are given a pointed coalgebra with $C = C_0 \oplus I$ where $I$ is a coideal, then the canonical projection $\pi_{C_0}: C \to C_0$ naturally provides a splitting. Hence, we occasionally use the notation $(C, I)$ to denote a pointed coalgebra with a splitting.
	
Lastly, we demonstrate that the category of pointed coalgebras is complete and cocomplete. We begin by presenting the following observations:

\begin{lemma} {\ }
\begin{enumerate}
\item Let $f: C \to D$ be a surjective coalgebra homomorphism. If $I$ is a coideal in $C$, then $f(I)$ is also a coideal in $D$.
\item If $f: C \to D$ is a surjective coalgebra homomorphism, where $C$ is pointed, then $D$ is also pointed, and $f(C_0) = D_0$.
\item The category of coalgebras over a field is complete and cocomplete, meaning it has both small limits and small colimits.
\end{enumerate}
\end{lemma}
	\begin{proof} {\ }
		\begin{enumerate} 
			\item A direct check: $\Delta(f(I)) = (f \otimes f) \Delta(I) \subset (f \otimes f) (I \otimes C + C \otimes I) \subset f(I) \otimes D + D \otimes f(I)$ and $\epsilon_D(f(I)) = \epsilon_C(I) = 0$. 
			\item First, $D_0 \subset f(C_0)$ if $f$ is surjective (Proposition 4.1.7 \cite{DR2}). Second, if $g$ is set-like in $C$, then $f(g)$ is also set-like and spans a simple subcoalgebra in $D$. Therefore, $f(C_0) \subset D_0$. 
			\item There are multiple approaches to prove this statement. One possible approach is discussed in \cite{limit} and \cite{limit2}.
		\end{enumerate}	
	\end{proof}
	
	\begin{theorem} {\ } 
		\begin{enumerate}
			\item Every pointed coalgebra is the filtered colimit of its finite-dimensional pointed subcoalgebra.
			\item The embedding $\iota$ from the category of pointed coalgebras to the category of coalgebras admits a right adjoint functor $\Phi$. 
			\item The category of pointed coalgebras is bicomplete.
		\end{enumerate}
	\end{theorem}
	\begin{proof}	{\ }
		\begin{enumerate}
			\item Any subcoalgebra of a pointed coalgebra is also pointed, and this can be proven using the same proof as the fundamental theorem for coalgebras. 
			\item 	We have an explicit construction given by the formula $\Phi(C) = \sum_{\alpha}C_\alpha$, where $C_\alpha$ runs over all pointed subcoalgebras of $C$. Note that $\Phi(C)$ is not only a subcoalgebra of $C$ but also a pointed coalgebra. This is because, by proposition 3.4.3 in \cite{DR2}, every simple subcoalgebra in $\Phi(C)$ must be in one of the $C_\alpha$ and thus has dimension one. Moreover, the image of a pointed coalgebra under a coalgebra homomorphism is also a pointed coalgebra. This follows from the fact that the image is a subcoalgebra and the map is surjective onto its image (we can use the previous lemma in this section). A map $f: C \to D$ induces a map $\Phi(f): \Phi(C) \to \Phi(D)$.
   
            Now We define the unit $\eta_X: X \to (\Phi\iota)(X)$ to be the identity on a pointed coalgebra $X$, and the counit $\varepsilon_Y: (\iota\Phi)(Y) \to Y$ as the natural inclusion of the pointed subcoalgebra $\Phi(Y)$ in a coalgebra $Y$. The maps $\eta$ and $\varepsilon$ are natural transformations. We need to show that the triangular identities hold.
            
            It is easy to check $(\Phi\iota\Phi)(X) = \Phi(X)$ and $(\iota\Phi\iota)(Y) = \iota(Y)$. In other words, the following triangular identities hold
			$$(\Phi(\varepsilon_X) \circ \eta_{\Phi(X)})(\Phi(X)) = \Phi(X)$$
			$$(\varepsilon_{\iota(X)} \circ \iota(\eta_X))(\iota(X))  = \iota(X)$$
			It follows from \cite{lane_2010} that there is an adjunction such that $\Phi$ is right adjoint to $\iota$.
			\item (2) shows that the category of pointed coalgebras is a coreflective subcategory of the category of coalgebras. Since the category of coalgebras is complete and cocomplete, it follows from the corollaries in \cite{herrlich_strecker_1971} that the category of pointed coalgebras is also complete and cocomplete.
		\end{enumerate}
	\end{proof}
 
\section{Algebraic perspective} 
In this section, we explore the algebraic structure of simply colored coalgebras. We start by examining the general simply colored comonoid in any monoidal abelian category. we show reduced comultiplication is coassociative. Subsequently, we narrow our scope to simply colored coalgebras in the category of $R$-modules. We prove the equivalence between simply colored coalgebras over a field and pointed coalgebras with a splitting.

\subsection{Simply colored comonoid}
In this subsection, we extend the constructions and proofs presented in Section 2.3 of \cite{kaufmann_mo_2022} to a monoidal category. We consider $\mathcal{C}$ to be a small monoidal category $(\mathcal{C}, \otimes, \mathds{1})$ that also possesses an abelian structure. It is important to note that, at this stage, we do not assume the compatibility of the monoidal structure with the abelian structure.

\begin{remark}
According to Mac Lane's Coherence Theorem, the associators and unitors can be assumed to be strict in the calculations. Additionally, the tensor product of multiple objects can be specified without ambiguity. Hence, in this paper, we assume any monoidal category is strict without loss of generality.
\end{remark}
	
\subsubsection{Reduced comultiplication}
Next, let $\delta: \mathbf{C} \to \mathbf{S}$ be a comonoid retraction between two comonoids $\mathbf{C}$ and $\mathbf{S}$ in the category. This means that $\delta$ is the left inverse of the monomorphism of comonoids $\iota: \mathbf{S} \to \mathbf{C}$. Since the category $\mathcal{C}$ is abelian, we have $\mathbf{C} \cong \mathbf{S} \oplus \mathbf{I}$, where $\mathbf{I}$ is the kernel of $\delta$.
We identify $\mathbf{S}$ with the image of $\iota$, and we consider $\delta$ as the composition $\iota \circ \delta$ (which becomes an idempotent). Consequently, $\mathbf{S}$ is a well-defined subcomonoid.
\begin{remark}
The notion of a subcomonoid is not generally well-defined. While $\Delta(\iota(\mathbf{S})) \subset \iota(\mathbf{S}) \otimes \iota(\mathbf{S})$, it may not be a subobject of $\mathbf{C} \otimes \mathbf{C}$. However, in our case, we require the map $\iota: \mathbf{S} \to \mathbf{C}$ to be a \textit{split} monomorphism of comonoids. The monoidal product of two split monomorphisms is also a split monomorphism, thanks to the interchange law. Hence, $\iota(\mathbf{S})$ is indeed a well-defined subcomonoid of $\mathbf{C}$ that is isomorphic to $\mathbf{S}$.
\end{remark}
	
	Consequently, the following equations make sense.
	\begin{equation}
		\delta \circ \delta = \delta
	\end{equation}
	\begin{equation}
		\epsilon \circ \delta = \epsilon
	\end{equation}
	\begin{equation}
		\Delta \circ \delta = (\delta \otimes \delta) \circ \Delta
	\end{equation}
	
	Now we will construct the following $\mathbf{S}$-bicomodule on $\mathbf{C}$ via $\omega_l = (\delta \otimes id) \circ \Delta$ and $\omega_r = (id \otimes \delta) \circ \Delta$. 
	\begin{proposition}
		$\omega_l = (\delta \otimes id) \otimes \Delta$ and $\omega_r = (id \otimes \delta) \otimes \Delta$ induces an $\mathbf{S}$-bicomodule on $\mathbf{C}$
	\end{proposition}
	\begin{proof}
		We need to check the following equation for the right comodule structure. 
		\begin{equation}
			(\omega_r \otimes id) \circ \omega_r =  (id \otimes \Delta) \circ \omega_r
		\end{equation}
		Do the computation:
		\begin{align*}
			(\omega_r \otimes id) \circ \omega_r &= (((id \otimes \delta) \circ \Delta) \otimes id) \circ (id \otimes \delta) \circ \Delta \\
			&= (((id \otimes \delta) \circ \Delta) \otimes \delta)  \circ \Delta \\ 
			&=  ((id \otimes \delta)\otimes \delta) \circ (\Delta \otimes id) \circ \Delta \\
			&= (id \otimes (\delta \otimes \delta)) \circ (id \otimes \Delta) \circ \Delta \\
			&= (id \otimes \Delta) \circ (id \otimes \delta) \circ \Delta \\
			&= (id \otimes \Delta) \circ \omega_r
		\end{align*}
		Note we have used the coassociativity and the fact $\delta$ is a comonoid homomorphism. 
		\begin{equation}
			(id \otimes \epsilon) \circ \omega_r = id
		\end{equation}
		Do the computation again:
		\begin{align*}
			(id \otimes \epsilon) \circ \omega_r &= (id \otimes \epsilon) \circ (id \otimes \delta) \circ \Delta \\
			&= (id \otimes (\epsilon \circ \delta)) \circ \Delta \\ 
			&= (id \otimes \epsilon) \circ \Delta \\
			&= id
		\end{align*}
		Here in $(\epsilon \circ \delta) = \epsilon$ we use the fact $\delta$ is a comonoid retraction. 
		With similar computation which we ignore, we show the following equations for the left comodule structure.
		\begin{equation}
			(id \otimes \omega_l) \circ \omega_l =  (\Delta \otimes id) \circ \omega_l
		\end{equation}
		\begin{equation}
			(\epsilon \otimes id) \circ \omega_l = id
		\end{equation}
		Finally, for the bicomodule structure, we need to compute
		\begin{equation}
			(\omega_l \otimes id) \circ \omega_r = (id \otimes \omega_r) \circ \omega_l
		\end{equation}
		\begin{align*}
			(\omega_l \otimes id) \circ \omega_r &= (((\delta \otimes id)\circ \Delta) \otimes id) \circ  (id \otimes \delta) \circ \Delta \\
			&= ((\delta \otimes id) \otimes id) \circ (\Delta \otimes id) \circ  (id \otimes \delta) \circ \Delta \\
			&=((\delta \otimes id) \otimes id) \circ (\Delta \otimes \delta) \circ \Delta \\
			&= ((\delta \otimes id) \otimes id) \circ ((id \otimes id) \otimes \delta) \circ (\Delta\otimes id) \circ \Delta \\
			&= (\delta \otimes (id \otimes id)) \circ (id \otimes (id \otimes \delta)) \circ (id \otimes \Delta) \circ \Delta \\
			&= (\delta \otimes (id \otimes id)) \circ (id \otimes ((id \otimes \delta) \circ \Delta)) \circ \Delta \\
			&=(\delta \otimes ((id \otimes \delta) \circ \Delta)) \circ \Delta \\
			&= (id \otimes ((id \otimes \delta) \circ \Delta)) \circ (\delta \otimes id) \circ \Delta \\
			&=(id \otimes \omega_r) \circ \omega_l
		\end{align*}
	\end{proof}
	
	\begin{remark}
		In the theory of commutative rings, if $B$ is an $A$-algebra, it follows the multiplication of $B$ is compatible with action of $A$. Analogously, the retraction-induced bicomodule is compatible with the comonoid structure. The compatibility means the following equations. 
		\begin{equation}
			(\omega_r \otimes id) \circ \Delta = (id \otimes \omega_l) \circ \Delta
		\end{equation}
		If the category $\mathcal{C}$ is abelian, this equation means the comultiplication factors through ``cotensor" over $\mathbf{S}$. We will define what we mean by cotensor later.  Now we compute:
		\begin{align*}
			(\omega_r \otimes id) \circ \Delta &= ((id \otimes \delta)\circ \Delta) \otimes id)\circ \Delta \\
			&= ((id \otimes \delta) \otimes id) \circ (\Delta \otimes id) \circ \Delta \\
			&= (id \otimes (\delta \otimes id)) \circ (id \otimes \Delta) \circ \Delta \\
			&= (id \otimes (\delta \otimes id) \circ \Delta) \circ \Delta \\
			&= (id \otimes \omega_l) \circ \Delta
		\end{align*}
		\begin{equation}
			(id \otimes \Delta) \circ \omega_l = (\omega_l \otimes id) \circ \Delta
		\end{equation}
		\begin{align*}
			(id \otimes \Delta) \circ \omega_l &= (id \otimes \Delta) \circ (\delta \otimes id) \circ \Delta \\
			&= (\delta \otimes \Delta) \circ \Delta \\
			&= (\delta \otimes (id \otimes id)) \circ (id \otimes \Delta) \circ \Delta \\
			&= ((\delta \otimes id) \otimes id) \circ (\Delta \otimes id) \circ \Delta \\
			&= (((\delta \otimes id) \circ \Delta) \otimes id) \circ \Delta \\
			&= (\omega_l \otimes id) \circ \Delta 
		\end{align*}
		In a similar computation, it follows
		\begin{equation}
			(id \otimes \omega_r) \circ \Delta = (\Delta \otimes id) \circ \omega_r
		\end{equation}
		Moreover, both $\omega_l$ and $\omega_r$ act on $\mathbf{S}$ as its own comultiplication. Therefore, we have more equations
		\begin{equation}
			(id \otimes \omega_r) \circ \omega_r = (id \otimes \omega_l) \circ \omega_r = (id \otimes \Delta) \circ \omega_r = (\omega_r \otimes id) \circ \omega_r
		\end{equation}
		We will compute some of equations and the rest is left to the readers.  
		\begin{align*}
			(id \otimes \omega_r) \circ \omega_r &= (id \otimes (id \otimes \delta) \circ \Delta) \circ (id \otimes \delta) \circ \Delta \\
			&= (id \otimes (id \otimes \delta)) \circ (id \otimes \Delta) \circ (id \otimes \delta) \circ \Delta \\
			&= (id \otimes (id \otimes \delta))\circ (id \otimes (\Delta \circ  \delta)) \circ \Delta \\
			&=  (id \otimes (id \otimes \delta))\circ (id \otimes (\delta \otimes  \delta)) \circ (id \otimes \Delta) \circ \Delta \\
			&= (id \otimes (\delta \otimes (\delta \circ \delta))) \circ (id \otimes \Delta) \circ \Delta \\
			&= (id \otimes (\delta \otimes (\delta))) \circ (id \otimes \Delta) \circ \Delta \\
			&=  (id \otimes (\delta \otimes (\delta))) \circ (id \otimes \Delta) \circ \Delta \\
			&=  (id \otimes (\delta^2 \otimes \delta)) \circ (id \otimes \Delta) \circ \Delta \\
			&=  (id \otimes (\delta \otimes id))\circ (id \otimes (\delta \otimes  \delta)) \circ (id \otimes \Delta) \circ \Delta \\
			&= (id \otimes (\delta \otimes id))\circ (id \otimes \Delta) \circ (id \otimes \delta) \circ \Delta\\
			&= (id \otimes ((\delta \otimes id))\circ \Delta) \circ (id \otimes \delta) \circ \Delta \\
			&= (id \otimes \omega_l) \circ \omega_r
		\end{align*}
		\begin{align*}
			(id \otimes \omega_r) \circ \omega_r &= (id \otimes (\delta \otimes (\delta \circ \delta))) \circ (id \otimes \Delta) \circ \Delta \\
			&= (id \otimes (\delta \otimes \delta)) \circ (id \otimes \Delta) \circ \Delta \\
			&= (id \otimes \Delta) \circ (id \otimes \delta) \circ \Delta \\
			&= (id \otimes \Delta) \circ \omega_r 
		\end{align*}
		
		Also, in a similar computation
		\begin{equation}
			(\omega_l \otimes id) \circ \omega_l = (\omega_r \otimes id) \circ \omega_l = (\Delta \otimes id) \circ \omega_l = (id \otimes \omega_l) \circ \omega_l
		\end{equation}
	\end{remark}
	
	\begin{definition}
		The \textit{reduced decomposition} associated to the retraction $\delta$ is defined as $\overline{\Delta} = \Delta - \omega_r - \omega_l$.
	\end{definition}
	
	If the category $\mathcal{C}$ is an abelian monodial category, that is,  the following is a homomorphism of abelian groups.
	$$\mathcal{C}(\mathbf{X}, \mathbf{Y}) \otimes_{\mathbb{Z}} \mathcal{C}(\mathbf{X'}, \mathbf{Y'}) \to \mathcal{C}(\mathbf{X} \otimes \mathbf{X'}, \mathbf{Y} \otimes \mathbf{Y'})$$
	In particular, it implies the tensor product is linear in both entries. We call this condition \textit{biadditive}. 
	The reduced decomposition is coassociative in the following sense. It is a generalization of Lemma 2.19 in \cite{kaufmann_mo_2022}
	
	\begin{theorem} \label{1}
		If $\otimes$ is biadditive, $(\overline{\Delta} \otimes id) \overline{\Delta} = (id \otimes \overline{\Delta}) \overline{\Delta}$
	\end{theorem}
	\begin{proof}
		The left hand side is \\
		$(\overline{\Delta} \otimes id) \circ \overline{\Delta} =
		(\Delta \otimes id) \circ \Delta - (\omega_r \otimes id) \circ \Delta - (\omega_l \otimes id) \circ \Delta -(\Delta \otimes id) \circ \omega_r + (\omega_r \otimes id) \circ \omega_r + (\omega_l \otimes id) \circ \omega_r 
		- (\Delta \otimes id) \circ \omega_l + (\omega_r \otimes id) \circ \omega_l + (\omega_l \otimes id) \circ \omega_l$ \\
		
		The right hand side is\\
		$(id \otimes \overline{\Delta}) \circ \overline{\Delta} =
		(id \otimes \Delta) \circ \Delta - (id \otimes \omega_r) \circ \Delta - (id \otimes \omega_l) \circ \Delta -(id \otimes \Delta) \circ \omega_r + (id \otimes \omega_r) \circ \omega_r + (id \otimes \omega_l) \circ \omega_r
		-(\Delta\otimes id) \circ \omega_l + (id \otimes \omega_r) \circ \omega_l + (id \otimes \omega_l) \circ \omega_l $ \\
		
		Note that $(\Delta \otimes id) \circ \Delta = (id \otimes \Delta) \circ \Delta$ by coassociavity. And most of the terms are cancelled by previous computation. We only need to check the following remaining terms.
		\begin{equation}
			(id \otimes \Delta) \circ \omega_l = (\omega_l \otimes id) \circ \Delta
		\end{equation}
		\begin{align*}
			(id \otimes \Delta) \circ \omega_l &= (id \otimes \Delta) \circ (\delta \otimes id) \circ \Delta  \\
			&= (\delta \otimes \Delta) \circ \Delta \\
			&= (\delta \otimes (id \otimes id)) \circ (id \otimes \Delta) \circ \Delta \\
			&= ((\delta \otimes id) \otimes id) \circ (\Delta \otimes id) \circ \Delta \\
			&= (((\delta \otimes id) \circ \Delta) \otimes id) \circ \Delta \\
			&= (\omega_l \otimes id) \circ \Delta
		\end{align*}
		By similar calculation as above
		\begin{equation}
			(id \otimes \omega_r) \circ \Delta = (\Delta \otimes id) \circ \omega_r
		\end{equation}
		After cancelling terms, we obtain $(\bar{\Delta} \otimes id) \bar{\Delta} - (id \otimes \bar{\Delta}) \bar{\Delta} = 0$ and the proof is completed
	\end{proof}
	
In this case, we will refer to $\bar{\Delta}$ as the \textit{reduced comultiplication}. There are two reasons for using the term "comultiplication." Firstly, it is coassociative. Secondly, when considering a coalgebra over a field, the retraction $\delta: C \to S$ leads to a splitting $C = S \oplus I$ for a certain coideal $I$. An elementary argument (which we will present later) shows that $I$ is closed under the reduced comultiplication induced by $I$. Consequently, $(I, \bar{\Delta})$ forms a coalgebra without a counit. \\
	
\subsubsection{Short discussion on simply colored cooperads} \label{33} {\ }
	
If $\mathds{1}$ is a unit object, it naturally possesses a comonoid structure induced by the left and right unitors $\mathds{1} \cong \mathds{1} \otimes \mathds{1}$, and the counit map is simply an isomorphism of comonoids. We can define the set-like comonoid $\mathbf{S}$ over a set $S$ in the category $\mathcal{C}$ as a direct sum of comonoids $\bigoplus_{s \in S}\mathds{1}_s$, where $S$ is a set and $\mathds{1} \cong \mathds{1}_s$ is an isomorphism of comonoids. It is expected that simply colored comonoids can be defined in any monoidal category with an abelian structure that may not necessarily be compatible. The challenge lies in finding a suitable description for the conilpotent condition. The following example of a simply colored cooperad serves as a starting point for the concept of simply colored comonoids. \\

We use the notations and definitions from \cite{LodayVallette}. The ground ring we consider is a field with characteristic zero. Recall that a cooperad is a comonoid in the category of $\mathbb{S}$-modules, which is both an abelian category and a monoidal category with $\bar{\circ}$ as the non-biadditive monoidal product. The construction we employ is similar to that of the conilpotent cooperad in \cite{LodayVallette}. We start with a cooperad $\mathbf{S}$ concentrated in arity one, where $\mathbf{S}(0) = C[S]$ is a set-like coalgebra over a set $S$. Next, we consider a cooperad $\mathbf{C}$ such that $\delta: \mathbf{C} \to \mathbf{S}$ is an operadic retraction. Similar to the coalgebra case, there exists a cooperadic coideal $\mathbf{I}$ such that $\mathbf{C} = \mathbf{S} \oplus \mathbf{I}$. We adopt Loday's method of defining the conilpotent condition, and we define the following concepts in order:
	$$\tilde{\Delta}^0 = id$$
	$$\tilde{\Delta}^1 = \Delta + (id \bar{\circ} \delta)\Delta$$
	Then, we can define it iteratively,
	$$\tilde{\Delta}^n = (id \bar{\circ} \tilde{\Delta})\tilde{\Delta}^{n-1}$$
	and 
	$$\hat{\Delta}^n = \tilde{\Delta}^n - (id \bar{\circ} ((id \bar{\circ} \delta)\Delta))\tilde{\Delta}^{n-1}$$
	This yields a filtration $F_0\mathbf{C} = \mathbf{S}$ and $F_n\mathbf{C} = \ker\hat{\Delta}^n$. 
	Then, we can define the simply colored cooperad is the pair $(\mathbf{C}, \delta)$ s.t. the filtration we just defined is exhaustive (our conilpotent condition). Note that by counit axiom, it can be shown that when $\mathbf{S} \cong \mathds{1}$ the monoidal unit , the right coaction $\omega_r = (id \bar{\circ} \delta)\Delta$ amounts to ``adding a level made up of the unique set-like $|$", which generalizes construction (see page 169 \cite{LodayVallette}) using $id \bar{\circ} \eta$ .

As a consequence of our generalization to arbitrary monoidal categories, one can extend the discussion for conilpotent cooperad in \cite{LodayVallette} to  simply colored cooperads, or even simply colored co-properads. This is due to the possibility to define operads and properads as plethysm monoids. So their dual structures are given by a comultiplication. The existence of such a plethysm construction more generally has been given in \cite{KMonaco22} in terms of a plus construction. The details of this connection will be studied in future work.

\subsection{Simply colored coalgebras}
We begin by introducing the concept of set-like elements.

\begin{definition} {\ }
\begin{enumerate}
\item An element $g$ in a coalgebra $C$ is referred to as \textit{set-like} if $\Delta(g) = g \otimes g \in C \otimes C$, and $\epsilon(g) = 1$.
\item The \textit{set-like} coalgebra over a set $G$ is a free module with the basis $G$, where $\Delta(g) = g \otimes g$ and $\epsilon(g) = 1$ for $g \in G$. In this paper, we will denote this coalgebra as $C[G]$.
\end{enumerate}
\end{definition}

It is important to note that $C[G] = \bigoplus_{g \in G} Rg$. There is a unique coalgebra structure on $R$ by considering $1$ as a set-like element since $R$ is unital. Recall that a retraction $\delta: C \to C[G]$ implies the existence of a well-defined reduced comultiplication $\overline{\Delta} = (id - \delta \otimes id - id \otimes \delta) \circ \Delta$. 

\begin{definition}
Let $C$ be a coalgebra, $G$ be a set of set-like elements in $C$, and $\delta: C \to C[G]$ be a retraction from $C$ to the set-like coalgebra over $G$. A \textit{simply colored coalgebra} is a triple $(C, G, \delta: C \to C[G])$ such that for every $x \in \ker \delta$, there exists an integer $N$ (which may depend on $x$) such that $\bar{\Delta}^N(x) = 0$.
\end{definition}

\begin{remark}
In the definition of a simply colored coalgebra $(C,G,\delta)$, we do not require $G$ to contain all set-like elements of $C$. However, the fact $G$ contains all set-like elements of $C$ follows from the conilpotent condition, as we will demonstrate in Corollary \ref{14}.
\end{remark}

\begin{example}
    A simple example of a simply colored coalgebra is a set-like coalgebra with a zero coideal. 
\end{example}

In principle, one can replace $C[G]$ with any coalgebra if the representation of that coalgebra is well-understood. Occasionally, we use the notation $C = C[G] \oplus I$ instead of $(C, G, \delta)$ (where $I = \ker \delta$) to represent a simply colored coalgebra.

The definition implies the following properties of set-like coalgebras and simply colored coalgebras.
	
	\begin{proposition} {\ }
		\begin{enumerate}
			\item Let $C = C[G] \oplus I$ be a simply colored coalgebra, then $ \overline{{\Delta}}(I) \subset I \otimes I \subset C \otimes C$. In fact, $I$ is a $C[G]$-subbicomodule of $C$.
			\item The tensor product of two set-like coalgebras equipped the canonical coalgebra structure \footnote{If $C$ and $D$ are two coalgebras, the tensor product $C \otimes D$ has a natural coalgebra structure: $$\Delta_{C \otimes D} = (id \otimes \tau_{12} \otimes id) \circ (\Delta_C \otimes \Delta_D)$$ and $$\epsilon_{C \otimes D} = \epsilon_C \otimes \epsilon_D$$ where $\tau_{12}(c \otimes d) = d \otimes c$ is a linear isomorphism on $C \otimes D$ to $D \otimes C$} is also set-like, and thus simply colored
			\item The tensor coalgebra of two simply colored coalgebras $C = C[G] \oplus I$ and $D = C[S] \oplus J$ is also a simply colored coalgebra.
		\end{enumerate}
	\end{proposition} 
	\begin{proof} {\ }
		\begin{enumerate}
			\item $I$ is a coideal, so $\Delta(I) \subset I \otimes C + C \otimes I = I \otimes C[G] + C[G] \otimes I + I \otimes I$. It is clear that $\bar{\Delta}(I) \subset I \otimes I$ and $(\delta \otimes id)\Delta(I) \subset C[G] \otimes I \subset C[G] \otimes C$ (and $(id \otimes \delta)\Delta(I) \subset I \otimes C[G] \subset C[G] \otimes C$.
			\item 	The tensor of two free modules is still free. So it suffices to show tensor of set-like elements is also set-like. We compute for $g$ set-like in $C$ and $s$ set-like in $D$:
			\begin{align*}
				\Delta_{C \otimes D}(g \otimes s) &= (id \otimes \tau_{12} \otimes id) \circ (\Delta_C \otimes \Delta_D)(g \otimes s) \\
				&= g \otimes s \otimes g \otimes s
			\end{align*}
			\item $C \otimes D = C[G] \otimes C[S] \oplus C[G] \otimes J \oplus I \otimes C[S] \oplus I \otimes J$. We showed $C[G] \otimes C[S]$ is also a set-like coalgebra. Then we need to show $C[G] \otimes J \oplus I \otimes C[S] \oplus I \otimes J = C \otimes J + I \otimes D$ is an coideal. The counit condition is straightforward. We finish the proof by checking the comultiplication requirement \newline $\Delta(C \otimes J) \subset C \otimes D \otimes C \otimes J + C \otimes J \otimes C \otimes D$ and $\Delta(I \otimes D) \subset I \otimes D \otimes C \otimes D + C \otimes D \otimes I \otimes D$. Therefore, $\Delta(C \otimes J + I \otimes D) \subset (C \otimes J + I \otimes D) \otimes (C \otimes D) + (C \otimes D) \otimes (C \otimes J + I \otimes D)$. 
		\end{enumerate}
	\end{proof}
	\begin{remark}
		 The statement in the above proposition is also true if $C[G]$ is replaced by any coalgebra $S$. 
	\end{remark}
	\begin{example}
		Every coaugmented conilpotent coalgebra is simply colored. Recall that an coaugmentation map is an coalgebra homomorphism $u: R \to C$ s.t. $\epsilon \circ u = id$. If ground ring $R$ is a field $k$, the condition $\epsilon \circ u = id$ is obvious and does not need to specify. This induces a splitting $C = Re \oplus \ker\epsilon$ in which $e = u(1)$. We claim $\ker\epsilon$ is an coideal. It is obvious that $\epsilon(\ker\epsilon) = 0$. It suffices to show $\Delta(\ker\epsilon) \subset \ker\epsilon \otimes C + C \otimes \ker\epsilon$. Let $x \in \ker\epsilon$. In general 
		$$\Delta(x) \in (Re \otimes Re) \oplus (Re \otimes \ker\epsilon) \oplus (\ker\epsilon \otimes Re) \oplus (\ker\epsilon \otimes \ker\epsilon)$$
		We only need the fact $\Delta(x)$ has no terms in $Re \otimes Re$. By the left counit identity $(\epsilon \otimes id) \circ \Delta (x) = x$, it follows there exists $r  \in R$ and $i \in \ker\epsilon$ such that $x = re + i$ and $\Delta(x) - r (e \otimes e) - (e \otimes i) \in (\ker\epsilon \otimes Re) \oplus (\ker\epsilon \otimes \ker\epsilon)$. However, $x \in \ker\epsilon$ implies $r = 0$. Thus, we show $\ker\epsilon$ is an coideal and the projection $\pi: C \to C/\ker\epsilon \cong Re$ is an retraction that agrees with our notion. The reduced comultiplication then becomes $\bar{\Delta}(x) = \Delta(x) - e \otimes x - x \otimes e$, which agrees with the usual notion of reduced comultiplication. 
	\end{example}
	
	The following proposition provides a major source of simply colored coalgebras.  
	\begin{proposition} \label{10}
		Any non-negatively graded coalgebra $C = \bigoplus_{i \geq 0} C(i)$ whose zeroth homogeneous component $C(0)$ is a set-like coalgebra $C[X]$ over some set $X$ is simply colored with the choice of the coideal $\bigoplus_{i > 0} C(i)$. We will call such graded coalgebras \textit{space-like}. 
	\end{proposition}
	\begin{proof} 
		It suffices to show the conilpotent condition. We prove by induction that if $x \in C(i)$ for $i > 0$, then $\bar{\Delta}^n(x) = 0$ for any $n \geq i$. The base case for $i = 1$ is straightforward. Suppose it is true for $i = k$. Let $x \in C(k+1)$, $\Delta(x) \in \sum_{i = 0}^{k+1} C(i) \otimes C(n-i)$. And $\bar\Delta(x) \in \sum_{i = 1}^{n-1} C(i) \otimes C(n-i)$ because the reduced comultiplication kills the parts $C(0) \otimes C(k+1)$ and $C(k+1) \otimes C(0)$. Each $C(i)$ will be eliminated by at most $i$-iteration or reduced comultiplication by induction hypothesis. Since reduced comultiplication is billinear and coassociative, it follows $x$ is eliminated by at most $k+1$-iterations of reduced comultiplication. 
	\end{proof}
	
\begin{example} {\ }
\begin{enumerate}
\item A connected coalgebra is simply colored. Examples of simply colored coalgebras include Connes-Kreimer's Hopf algebra, Goncharov's Hopf algebra for multiple zeta values, and Baues' Hopf algebra for double loop spaces \cite{IRA1}. These three Hopf algebras can be unified categorically through the framework of Feynman Categories \cite{IRA1}. Specifically, if a Feynman category is factorization finite and almost connected, the associated Hopf algebra (as shown in Theorem 1.1 of \cite{IRA2}) is connected and, therefore, simply colored. The three mentioned Hopf algebras can be constructed by selecting appropriate almost connected and factorization finite Feynman categories.

\item The path coalgebra associated with a quiver is simply colored. This is because it is known to be a (coradically) graded coalgebra whose zero homogeneous components are spanned by set-like elements.

\item The categorical coalgebra arising from a Feynman category \cite{kaufmann_mo_2022} (see Theorem 5.3) is simply colored. A similar notion of categorical coalgebra in \cite{catcoalg} is also simply colored.

\item The component coalgebra \cite{MayMoreConcise} is simply colored.

\end{enumerate}
\end{example}
	
	\begin{example}
		A non-trivial example of simply colored coalgebra. Let $C = \mathbf{Z} \oplus \mathbf{Z} \oplus \mathbf{Z}/4\mathbf{Z}$ be a $\mathbf{Z}$-module. We give a coalgebra structure on $C$ by setting $g = (1, 0, 0)$ and $h= (0, 1, 0)$ two different set-like elements and $x = (0, 0, 1)$ a $g$-primitive elements. $\mathbf{Z}g \oplus \mathbf{Z}h$ is a well-defined subcoalgebra of $C$ (because the natural inclusion splits) and there is a natural splitting of coalgebra onto the subcoalgebra just by projection where the coideal is precisely generated by $x$. The key is that we allow the existence of torsion elements in the coalgebra.
	\end{example}
	
	\begin{definition}
		Given $C$ a coalgebra and $A$ an algebra, the $R$-module $\Hom(C, A)$ is a $R$-algebra with unit $\eta \circ \epsilon$ and multiplication $\star$ given by $(f \star g) (x) = m \circ (f \otimes g) \circ \Delta (x) = \sum_{x}\,f(x_{(1)}) g(x_{(2)})$. We say $\{\Hom(C, A), \star, \eta \circ \epsilon\}$ is \textit{the convolution algebra}.  
	\end{definition}

	If for $f,g \in \Hom(C, A)$ there is $f \star g  = g \star f = \eta \circ \epsilon$, we say $f$ is $\star$-inverse (or convolution inverse) to $g$ and vice versa. The antipode of the Hopf algebra $H$ can be equivalently formulated as the unique convolution inverse to the identity $id$ in $\Hom(H, H)$.
	
	The following theorem is a modification of Takeuchi's lemma used in \cite{kaufmann_mo_2022} and it allows one to simplify the problem of finding convolution inverse.
	\begin{theorem} \label{2}
		Let $F_0$ and $C$ be two coalgebras such that $C_0$ is a direct summand of $C = F_0 \oplus M$ as $R$-module and $C_0$ has a unique subcoalgebra structure inherit from $C$. Suppose for any $x \in C$, there exists $N$ s.t. its image under the composition of canonical projection onto $F_0$ with (the kernel is $M$) and iterated comultiplications is zero. 
		\[\begin{tikzcd}
			C && {C^{\otimes N}} && {M^{\otimes N}}
			\arrow["{\pi^{\otimes N}}", from=1-3, to=1-5]
			\arrow["\Delta^N", from=1-1, to=1-3]
		\end{tikzcd}\]Then the following conditions are equivalent for any element $f$ of $Hom(C, A)$ with an algebra $A$:
		\begin{enumerate}
			\item $f$ is convolution invertible
			\item $f|_{F_{0}}$ is convolution invertible in $Hom(F_{0}, A)$
		\end{enumerate}
	\end{theorem}
	When $C$ is an an coalgebra over a field and $F_0$ is the coradical, this is known as Takeuchi's lemma \cite{Takeuchi}. The proof for the modified version is the same as the one in theorem 3.3 \cite{kaufmann_mo_2022}. The key is we translate the condition $\Ext(C/F^{QT}_C, A) =0$ to the wedge product characterization.

	\begin{proposition} \label{3}
		Let $C = C[G] \oplus I$ be a simply colored coalgebra and $\pi: C \to I$ be the canonical projection induced by splitting. We have $\pi^{\otimes N}( \bar{\Delta}^N(x)) = \pi^{\otimes N}({\Delta}^N(x))$ for $x \in I$.  
	\end{proposition}
	\begin{proof}
		It follows because ${\Delta}(x) - \overline{\Delta}(x) \in C[G] \otimes C + C \otimes C[G]$ and the difference will be killed when passed to the quotient. ${\Delta}^N(x) - \bar{\Delta}^N(x)$ will be also eliminated when passed to the quotient by the same reason. 
	\end{proof}
	It follows that
	\begin{corollary} \label{14} {\ }
		\begin{enumerate}
			\item If $C = C[G] \oplus I$ is simply colored, then $G$ is the set of all set-likes.
			\item Given algebra $A$, any $f \in \Hom(C, A)$ is convolution invertible if $f|_{C[G]} \in \Hom(C|_{C[G]}, A)$ is convolution invertible.
			\item We call an bialgebra is simply colored if its coalgebra structure is simply colored. The simply colored bialgebra has an/the antipode if and only if the set of set-likes form a group under the multiplication. 
		\end{enumerate}
	\end{corollary} 
	\begin{proof}
		We only need to show (1). Let $t$ be an set-like not in $C[G]$. It cannot be in $I$ as well since $\epsilon(g) \neq 0$. So it has to be the form of $r + i$ for $ 0 \neq t \in C[G]$ and $0 \neq i \in I$. Since $\Delta^n(t) = t^{\otimes n}$, $\pi^{\otimes n}( \bar{\Delta}^n(t)) = \pi^{\otimes n}({\Delta}^n(t)) = i^{\otimes n} \neq 0$ for any $n \geq 0$. However, the conilpotent condition implies that every $i \in I$ can be eliminated by finite iterations of reduced comultiplication. Thus, $\pi^{\otimes N}( \bar{\Delta}^N(t)) = \pi^{\otimes N}(\bar{\Delta}^N(r)) + \pi^{\otimes N}(  \bar{\Delta}^N(t)) =\pi^{\otimes N}( \bar{\Delta}^N(t)) = 0$ for $N$ sufficiently large, which is an contradiction. 
	\end{proof}
	
	\subsubsection{Simply colored coalgebra over a field is equivalent to pointed coalgebra with a splitting}

  \begin{theorem}\label{11} {\ }
     Let $k$ be a field, 
     \begin{enumerate}
         \item Every simply colored coalgebra $(C, G, \delta)$ over $k$ can be viewed as a pointed coalgebra with a splitting $(C, \delta)$.
         \item Conversely, if $(C, \delta)$ is a pointed coalgebra with a splitting, it naturally gives rise to a simply colored coalgebra structure $(C, G, \delta)$, where $G$ is the set of all set-like elements in $C$.
     \end{enumerate}
 \end{theorem}  
	\begin{proof} {\ }
		\begin{enumerate}
			\item Let $C = C[G] \oplus I$ where $I = \ker \delta$. We show the coalgebra filtration formed by $\{\wedge^n(C[G])\}_{n \geq 0}$ is exhaustive. It follows directly from Proposition \ref{3}: if $ \bar{\Delta}^n(x) = 0$, then $x \in \wedge^n(C[G])$. By Exercise 4.1.9 in \cite{DR2}, the coradical $C_0 \subset C[G]$. Therefore, every cosimple subcoalgebra is one dimensional and $C$ is pointed.
			\item The first statement follows again from the fact every pointed coalgebra is isomorphic to a subcoalgebra of a path coalgebra associated to a quiver (proposition \ref{15}) and the fact subcoalgebra of coradically graded coalgebra is coradically graded (Exercise 4.4.6 \cite{DR2}). Now take any splitting $C = C_0 \oplus I$ of pointed coalgebras. The reduced comultiplication is induced by this splitting. By discussion of proof of theorem 5.4.1 in \cite{Susan}, it follows $\overline{\Delta}(I_n) \subset I_{n-1} \otimes I_{n-1}$ where $I_n = C_n \cap I$. Since $C_0 \cap I = \emptyset$, every element in $I$ is eliminated through some iterations of reduced comultiplication. 
		\end{enumerate}
		
	\end{proof}
	
	Therefore, we see simply colored coalgebra and pointed coalgebra with a splitting are the same thing. Since the subcoalgebra of pointed coalgebra over a field is well-defined and also pointed. It follows immediately that 
	\begin{corollary} {\ }
		\begin{enumerate}
			\item the subcoalgebra of simply colored coalgebra over a field can be realized to be a simply colored coalgebra over a field.
			\item A simply colored coalgebra over a field is a sum of finite dimensional simply colored subcoalgebras. Thus, every simply colored coalgebra over a field is a filtered limit of finite dimensional simply colored coalgebras. 
		\end{enumerate}
	\end{corollary}
\begin{proof} {\ }
\item    Consider a subcoalgebra $D$ of a pointed coalgebra with a splitting $C = C[G] \oplus I$. According to Proposition 3.4.3 \cite{DR2}, we have $D \cap C_0 = D \cap C[G] = D_0$. Thus, $D_0 \subseteq C[G]$ is also pointed, and in particular, a set-like coalgebra since $C[G]$ is cosemisimple. The splitting restricts to $D$ as $\delta|_{D}: D \to D_0$, resulting in a pointed subcoalgebra with a splitting, denoted as $(D, \delta|_{D})$.

\item By the fundamental theorem of coalgebras, every element $x$ in the coalgebra $C$ belongs to a finite-dimensional subcoalgebra. Moreover, any subcoalgebra of a pointed coalgebra is also pointed. Therefore, for any element $x$, there exists a finite-dimensional pointed subcoalgebra with a splitting where $x$ resides. Consequently, the pointed coalgebra $(C, I)$ with a splitting can be regarded as the sum of all finite-dimensional pointed subcoalgebras with a splitting.
	\end{proof}
	\begin{example}
		As a result, the pointed curved coalgebra in \cite{ckd} is also a simply colored coalgebra. 
	\end{example}
	
	\subsubsection{How is simply colored coalgebra `colored'?} \label{12} 
The question has already been addressed in \cite{kaufmann_mo_2022} through the use of orthogonal projectors. This method involves a specific choice of an orthogonal idempotent family, which has been extensively studied in the context of pointed coalgebras (e.g., see \cite{pointed} and Chapter 4 of \cite{DR2}). In this section, we provide a summary of the results presented in the aforementioned papers and books.

When $A = R$, the convolution algebra $C^* = \Hom_R(C, R)$ is referred to as the \textbf{dual algebra} of $C$. The multiplication in $C^*$ is inherited from the multiplication in $R$, thus we omit the star notation. Any right $C$-comodule $M$, which is also an $R$-module, induces a left $C^*$-module structure on $M$ defined as $f \cdot m = (id \otimes f) \rho_r(m)$. Similarly, any left $C$-comodule $M$ gives rise to a right $C^*$-module structure on $M$ defined as $m \cdot f = (f \otimes id) \rho_l(m)$.
	
\begin{definition} \cite{DR2}
		Let $C$ be an coalgebra, An \textbf{orthogonal family of idempotents} $\mathcal{E}$ of $C^*$ is a family of element $\{e_i\}_{i \in I}$ of $C^*$ s.t. $e_ie_i = e_i$ and $e_ie_j = 0$ for $i \neq j$. Furthermore, a orthogonal family of idempotents $\mathcal{E}$ is called \textbf{orthonormal} if the counit $\epsilon = \sum_{i \in I}e_i$ \footnote{By the fundamental theorem of coalgebras, a coalgebra is a direct limit of finite-dimensional (sub)coalgebras. Thus, its dual becomes an inverse limit of algebras, which possesses a topology that allows for certain infinite sums to make sense. This algebra is sometimes referred to as a \textit{pseudocompact} algebra in the literature.}.
	\end{definition}

	Denote $L_{e}(c) = e \cdot c$ and $R_{e^g} = c \cdot  e$ for $e \in \mathcal{E}$. The following identites are true (see pp 132 and pp 110 in \cite{DR2} for detail). It still works if the ground ring is only assumed to be commutative unital.
	\begin{lemma} \label{5} \cite{DR2} {\ }  
		\begin{enumerate}
			\item $L_e \circ L_{e'} = \delta_{e, e'}L_e$ and $R_e \circ R_{e'} = \delta_{e, e'}R_e$
			\item $\sum_{e \in \mathcal{E}}L_e = I_C = \sum_{e \in \mathcal{E}}R_e$ where $I_C$ is the identity action on $C$ of $C^*$.
			\item $(id \otimes L_{e}) \circ \Delta = \Delta \circ L_{e}$ and $(R_{e} \otimes id) \circ \Delta = \Delta \circ R_{e}$ 
			\item $\sum_{e \in \mathcal{E}}(L_e \otimes R_e) \circ \Delta = \Delta$
			\item $L_e \circ R_{e'} = R_{e'} \circ L_e$
		\end{enumerate}
	\end{lemma}

For a simply colored coalgebra $C = C[G] \oplus I$, we can choose the orthogonal family to be ${\delta_{g}}_{g \in G}$, where $e_g \in C[G]^*$ represents the dual basis to $g$. This choice is sensible due to the hypothesis that $C[G]$ is free. It can be verified that this family is also orthonormal. As a consequence of the previous lemma, we have the following proposition:

\begin{proposition}\label{4}
An $C[G]$-bicomodule $M$ can be decomposed as $M = \bigoplus_{g, g' \in G}{^gM^{g'}}$, where $^gM^{g'} = e_{g'} \cdot M \cdot e_g$. Moreover, it has an explicit description ${^gM^{g'}} = {m \in {^gM^{g'}}: \rho_l(m) = g \otimes m, \rho_r(m) = m \otimes g'}$.
\end{proposition}

The key observation is that every left/right/bi-comodule over $C[G]$ is cosemisimple. In particular, a $C[G]$-bicomodule is equivalent to a $G$-bigraded module.

\begin{remark}
Alternatively, one can directly prove the result using the definition of a bicomodule. However, it is important to note that in this case, $ {^g}M \cap M^{g'} = {^gM^{g'}}$ holds due to the orthonormal condition. It is not true that the equality holds in general, and one cannot deduce it from the same argument presented in Example 1.6.7 in \cite{Susan}. The conditions of freeness of $C[G]$ and being a bicomodule are essential for the result.
\end{remark}
	
	Move back to the main discussion. The splitting $C = C[G] \oplus I$ endows $C$ with a $C[G]$-bicomodule structure and give rise to a reduced comultiplication. 
	\begin{proposition}
		Given a splitting $C = C[G] \oplus I$, where $I$ is an coideal, it follows
		\begin{enumerate}
			\item $\Delta({^{g'}C^g}) \subset \sum_{s \in G}{^{g'}C^s} \otimes {^{s}C^g}$
			\item $I = \oplus_{g, g' \in G}({^{g'}C^g} \cap I)$ and $I \cap {^{g'}C^g} = {^{g'}C^g}$ for $g' \neq g$
			\item As a result of (1) and (2), for $ x \in {^{g'}C^g}$, we have both $\Delta(x) - g' \otimes x \in I \otimes C$ and $\Delta(x) - x \otimes g \in C \otimes I$. 
			
		\end{enumerate}
	\end{proposition}
	\begin{proof} {\ }
		\begin{enumerate}
			\item It follows from Lemma \ref{5} and Proposition \ref{4}. 		
			\item For $i \in I$, we know $i = \epsilon \cdot  i \cdot  \epsilon$. Let $i_{gg'} = e_g' \cdot  i \cdot  e_g$ so $i = \sum_{g, g' \in G}i_{gg'}$ a direct sum where $i_{gg'} = e_g(i_1)e_{g'}(i_2)i_3$. Note if $g' \neq g$, $i_{gg'} \in g \cdot  (C[G]) \cdot  g' \oplus g \cdot  I \cdot  g' = g \cdot  I \cdot  g'$. If $g = g'$, then $i_{gg'} = c_g g \oplus i'_{gg'}$ where $g \in G$ and $i'_{gg'} \in g \cdot  I \cdot  g' \subset I$. $C[G] \oplus I = C$ implies $c_g = 0$ for all $g$. We conclude then $I = \oplus_{g, g' \in G} (I \cap {^{g'}C^g})$. To show ${^{g'}C^g} \subset I \cap {^{g'}C^g}$, just apply $\epsilon = \sum_{s \in G}e_s$ to $e_{g'} \cdot  x \cdot  e_{g}$ for $x \in C$ and check $\sum_{s \in G}e_{g'}(x_1)e_s(x_2)e_g(x_3) = 0$ by orthogonality. 
			\item Observe $\Delta({^{g'}x^g}) - g' \otimes {^{g'}x^g}  \in \sum_{s \in G; s\neq g'}{^{g'}C^s} \otimes {^{s}C^g}$ and $\epsilon({^{g'}C^s}) = 0$. Thus $\Delta(x) - g' \otimes x \in I \otimes C$. The other is similar. 
		\end{enumerate}	
	\end{proof}
	
	We show a simply colored coalgebra has the following structure:
	\begin{itemize}
		\item $C$ is a $G$-bigraded module: $C = \bigoplus_{g, h \in G} {^gC^h}$ and $I$ is a $G$-bigraded submodule of $C$ that $I = \oplus_{g, h \in G}({^{g}C^h} \cap I)$. Both are compatible with both comultiplication and the retraction:${^gC^g} = Rg \oplus {^gI^g}$, if $g \neq h$ then ${^gC^h} = {^gI^h}$ and $\Delta({^gC^h}) \subset \sum_{t \in G}{^gC^t}\otimes {^tC^h}$.
		\item There is a well-defined coassociative reduced comultiplication $\bar{\Delta}$ induced by $\delta$ s.t. For $x \in {^gC^h}$, $\bar{\Delta}(x) = \Delta(x) - x \otimes h - g \otimes x$. Furthermore, for any $i \in I$, there exists a non-negative integer $n$ (that depends on x) s.t. $\bar{\Delta}^n(x) = 0$. 
	\end{itemize}
	The algebraic structure also determines a simply colored coalgebra. We will take a categorical viewpoint on the algebraic structure of simply colored coalgebras in the next section.

	\section{Categorical perspective}	
In this section, we explore the categorical perspective of simply colored coalgebras. The key result we establish is that the category of simply colored coalgebras over a field (equivalently, pointed coalgebras with a splitting) admits all small limits and colimits. The main challenge lies in demonstrating that the category has both products and coequalizers. To overcome this challenge, we introduce the notion of reduced colored coalgebras, which bears resemblance to the approach used in studying conilpotent coaugmented coalgebras by disregarding the coaugmentation.
		
	\subsection{Three categories related to simply colored coalgebras}
	Note for this part we do not assume the ground ring is a field. 
	\begin{definition}
		The category of simply colored coalgebras is defined as follows:
		\begin{itemize}
			\item The objects are precisely simply colored coalgebras $(C, G, \delta)$.  
			\item a morphism between $(C, G, \delta)$ and $(D, H, \mu)$ is precisely a coalgebra homomorphism $f:C \to D$ such that $\mu \circ f = f|_{C[G]} \circ \delta$. 
		\end{itemize}
	\end{definition}
\begin{remark}
We have shown that the image of a set-like element under a morphism between simply colored coalgebras is also a set-like element. Therefore, the map $f|_{C[G]}$ has the codomain $C[H]$. Additionally, according to the fundamental theorems of coalgebra homomorphism \cite{Sweedler}, a morphism between $(C, G, \delta)$ and $(D, H, \mu)$ in the category of simply colored coalgebras can be equivalently described as a coalgebra homomorphism $f$ between $C$ and $D$ such that $f(\ker \delta) \subseteq \ker \mu$.
\end{remark}

Recall that we have shown how the reduced comultiplication makes a simply colored coalgebra $C = C[G] \oplus I$ to a coalgebra without counit $(I, \overline{\Delta})$. This coalgebra without counit possesses the property that every element in $I$ is eliminated by a finite number of iterations of the reduced comultiplication. Specifically, it is reduced in the following sense:

\begin{definition}
		A coalgebra without counit $\overline{C}$ is called \textit{reduced} simply colored coalgebra if
		\begin{itemize}
			\item It is conilpotent: for every $x \in \overline{C}$, there exists a finite natural number $n$ s.t. $\overline{\Delta}^n(x) = 0$. \footnote{Since counit axiom does not hold, the comultiplication is no longer guaranteed to be injective.}
			\item There exists a set $G$ and $\overline{C}$ is a $C[G]$-bicomodule. In particular, $\overline{C}$ is a naturally $G$-bigraded comodule $\overline{C} = \bigoplus_{g, h \in G} {^g\overline{C}{}^h}$.
			\item Furthermore, $\overline{C}$ is $G$-bigraded coalgebra over $R$: $\overline{C} = \bigoplus_{g, h \in G} {^g\overline{C}{}^h}$ and $\Delta({^g\overline{C}{}^h}) \subset \oplus_{s \in S} {^g\overline{C}{}^s} \otimes {^s\overline{C}{}^h}$.
		\end{itemize}
		The reduced colored coalgebra is then characterized as a pair $(\overline{C}, G)$.
\end{definition}

\begin{definition}
    The category of reduced (simply) colored coalgebras is defined as follows:
	\begin{itemize}
	       \item The objects are precisely reduced colored coalgebras $(\overline{C}, G)$.
			\item The morphism are precisely pairs $(\overline{f}, i)$  where $\overline{f}: \overline{C} \to \overline{D}$ is a coalgebra homomorphism that only preserves comultiplication and $i:G \to S$ is a morphism of sets.
		\end{itemize}
	\end{definition}

There is a third category $(V, G)$ formed by forgetting the coalgebra structure of the reduced colored coalgebras. 
    \begin{definition}
	The category of colored modules (over a ring $R$) has the following:
		\begin{itemize}
			\item The objects are precisely pairs $(V, G)$ where $V$ is a $G$ bigraded $R$-module with $V = \bigoplus_{g, h \in G} {^gV^h}$ 
			\item The morphisms are precisely pairs $(f, i):(V, G) \to (W, S)$ where $i: G \to S$ is a morphism of sets and $f({^gV^h}) \subset {^{i(g)}W^{i(h)}}$ \footnote{In other word, $(f, i) = \bigoplus_{g, h \in G}f|_{^gV^h}$ and $\im(f|_{^gV^h}) \subset {^{i(g)}W^{i(h)}}$. On the other hands, if we have a sequence of maps $\phi_{g, h}$ on ${^gV{}^h}$ s.t. $\im(\phi_{g, h}) \subset {^{i(g)}W^{i(h)}}$ given a function of sets $i: G \to S$. Then we can construct a morphism of colored modules via gluing $(\phi, i) = \bigoplus_{g, h \in G}\phi_{g, h}$.}.
		\end{itemize}
	\end{definition}
 
What are the relationships between these categories?  In the upcoming subsection, we will see in the next part that there exists a cofree functor from the category of colored vector spaces (colored modules over a field) to the category of pointed coalgebras with a splitting (simply colored coalgebras over a field). Unfortunately, this adjunction may not hold for over arbitrary ring $R$. For the rest of the subsection, we will establish an equivalence between the category of simply colored coalgebras and the category of reduced colored coalgebras. As discussed earlier, we already know how to obtain a reduced colored coalgebra from a simply colored coalgebra. Next, we will show the reverse process by extending the reduced comultiplication $\overline{\Delta}$. The procedure is as follows:

	\begin{itemize}
		\item Join the set-like coalgebra $C[G]$ over the set $S$ as a direct summand to $\overline{C}$, so $C = C[S] \oplus \overline{C}$.
		\item Extend the counit map $\epsilon$ from $C[G]$ by setting $\epsilon(\overline{C}) = 0$
		\item The comultiplication $\Delta$ on $C[G] \oplus \overline{C}$ is defined as follows: its restriction on $C[S]$ is the same as comultiplication on set-like coalgebra. Its restriction on $\overline{C} = \bigoplus_{g, h \in G}{^g\overline{C}{}^h}$ is constructed as $\Delta(x) = \rho_l(x) + \overline{\Delta}(x) + \rho_r(x)$ for $x \in \overline{C}$. It is helpful to read component-wise that $\Delta(x) = g \otimes x + \overline{\Delta}(x) + x \otimes h$ for $x \in {^g\overline{C}^h}$.
	\end{itemize}
	We need to show the resulting construction is indeed a coalgebra, and furthermore, a simply colored coalgebra.
	\begin{proposition} \label{6}
		The construction $C[G] \oplus \overline{C}$ described above produces a simply colored coalgebra $C = C[G]\oplus I$ where $I = \overline{C}$ is an coideal of $C$.
	\end{proposition}
	\begin{proof}
		First, we need to verify both counit and coassociativity axioms are satisfied. The counit axiom is straightforward once the constructed comultiplication is well-defined, whereas the coassociativity requires computations. It is enough to check component-wise. For simplicity, we use the notation ${^gx^h} \in {^g\overline{C} ^h}$ and $\overline{\Delta}(^gx^h) = \sum_{s \in G}{^gx^s} \otimes {^sx^h}$. 
		$$(\Delta \otimes id)\Delta({^gx^h}) = (\Delta \otimes id)(g \otimes {^gx^h} + \sum_{s \in G}{^gx^s} \otimes {^sx^h} + {^gx^h} \otimes h) $$
		$$= g \otimes g \otimes {^gx^h} + \sum_{s \in G}(g \otimes {^gx^s} + \sum_{t \in G}{^gx^t} \otimes {^tx^s} + {^gx^s} \otimes s)\otimes {^sx^h} + $$
		$$ (g \otimes {^gx^h} + \sum_{t \in G}{^gx^t} \otimes {^tx^s} + {^gx^h} \otimes h) \otimes h $$
		and
		
		$$(id \otimes \Delta)\Delta({^gx^h}) = (id \otimes \Delta)(g \otimes {^gx^h} + \sum_{s \in G}{^gx^s} \otimes {^sx^h} + {^gx^h} \otimes h) $$
		$$= g \otimes (g \otimes {^gx^h}) + \sum_{s \in G}{^gx^s} \otimes {^sx^h} + {^gx^h} \otimes h) + $$
		$$  \sum_{s \in G}{^gx^s} \otimes  (s \otimes {^sx^h} + \sum_{{t'} \in G}{^sx^{t'}} \otimes {^{t'}x^h} + {^sx^h} \otimes h)  + {^gx^h} \otimes h\otimes h$$
		Observe that due to coaasociativity of $\overline{\Delta}$\footnote{The Sweddler's notation is omitted, but it is easy to see this equation corresponds to $(\overline{\Delta} \otimes id) \overline{\Delta}({^gx^h}) = (id \otimes \overline{\Delta})\overline{\Delta}({^gx^h})$}, 
		$$\sum_{s \in G}\sum_{t \in G}{^gx^t} \otimes {^tx^s} \otimes {^sx^h} =  \sum_{s \in G}\sum_{{t'} \in G} {^gx^s} \otimes {^sx^{t'}}
		\otimes {^{t'}x^h}$$
		Then it follows that $(\Delta \otimes id) \circ \Delta = (id \otimes \Delta) \circ \Delta$.
		Finally, $I = \overline{C}$ is a coideal can be checked easily. 
	\end{proof}
		
To extend the construction of Proposition \ref{6} to a functor, we now define how it acts on each morphism.

	\begin{proposition} \label{7}
		A coalgebra homomorphism between simply colored coalgebras induces an homomorphism of coalgebras without counit. On the other hand, A coalgebra homomorphism without counit between reduced colored coalgebras has an extenstion to a coalgebra homomorphism to the simply colored coalgebras.
	\end{proposition}
	\begin{proof}
		The first statement is clear from the structure of simply colored coalgebras because coalgebra homomorphism maps set-like elements to set-like elements , $i(g) = f(g)$ for $f: C \to D$ and $\overline{f} = f|_{\overline{C}}: \overline{C} \to \overline{D}$.\\ For the second statement, it is due to the way we construct the reduced comultiplication. If $(\overline{f}, i): (\overline{C}, G) \to (\overline{D}, S)$ between two reduced colored coalgebras, then there is an extension to the morphism between simply colored coalgebras as $f(g) = i(g)$ for $g \in G$ and $f(x) = \overline{f}(x)$ for $x \in \overline{C}$. We check it is well-defined: 
		\begin{itemize}
			\item $\Delta(f(x)) = i(g) \otimes \overline{f}(x) + \overline{f}(x') \otimes \overline{f}(x'') + \overline{f}(x) \otimes i(h)$ for $x \in {^g\overline{C}{}^h}$ where $\overline{f}(\overline{\Delta}(x)) = (\overline{\Delta} \otimes \overline{\Delta})\overline{f}(x) = \overline{f}(x') \otimes \overline{f}(x'')$. In this case, the $x \in {^g\overline{C}{}^h}$ is sent to $\overline{f}(x) \in {^{i(g)}\overline{C}{}^{i(h)}}$. Moreover, $\Delta(f(g)) = \Delta(i(g)) = i(g) \otimes i(g)$.
			\item $\epsilon(f(x)) = 0$ if $x \in \overline{C}$ and $\epsilon(f(g)) = \epsilon(i(g)) = 1$
		\end{itemize}
	\end{proof}  
	
	\begin{corollary}
		There is a functor $F$ from the category of simply colored coalgebras to the category of reduced colored coalgebras and a functor $G$ that goes in the other direction such that $F$ (and $G$) gives an isomorphism of categories.
	\end{corollary}
	\begin{proof}
		Define $F((C, I)) = (I, G)$ ($C = C[G] \oplus I$) and $F(f) = (f|_{I}, f|_{C[G]})$. Also, define $G$ based on Proposition \ref{6} and Proposition \ref{7}. It is straightforward to check both $F$ and $G$ are functors and $F \circ G = \Id$ and $G \circ F = \Id$.
	\end{proof}
	\begin{remark}
		From now on, we will write a simply colored coalgebra as $C = C[G] \oplus \overline{C}$ where $\overline{C}$ is the associated reduced colored coalgebra. 
	\end{remark}

	\subsection{Interlude: cofree construction}
From this point forward, we assume that the ground ring is a field. As established by previous theorems, a simply colored coalgebra is equivalent to a pointed coalgebra with a splitting. In his work \cite{pointed}, Radford demonstrates that the forgetful functor from the category of pointed coalgebras with a splitting to the category of colored vector spaces (colored bicomodules over $k$) admits a left adjoint. This left adjoint is known as the cofree pointed coalgebra (with a splitting) over a pair $(V, G)$, and it is constructed using a cotensor coalgebra. Let us recall the notion of the cotensor product.

Given a right $C'$-comodule $M$ with a right coaction $\rho_R$ and a left $C'$-comodule $N$ with a left coaction $\rho_L$, we define the \textit{cotensor product} $M \square_{C'} N$ as the equalizer of $id_M \otimes \rho_L$ and $\rho_R \otimes id_N$:
	\[
	\begin{tikzcd}
		M \square_{C'} N \arrow[r, dashed] & M \otimes N \arrow[r] \arrow[r] \arrow[r] \arrow[r] \arrow[r] \arrow[r, "\rho_r \otimes id_N"] \arrow[r, "id_M \otimes \rho_l"', shift right] & M \otimes C' \otimes N
	\end{tikzcd}
	\]
	Since $M \otimes N$ is a vector space, $M \square_{C'} N$ is the kernel of $(\rho_R \otimes id_N - id \otimes \rho_L)$ and subvector space of $M \otimes N$. When the underlying space is a vector space, there is an associativity $(M \square_{C'} N) \square_{C'} Q \cong M \square_{C'} (N \square_{C'} Q)$. Next, we can construct the cotensor coalgebra (\cite{nichols_1978}) over a pair $(S, M)$ of a coalgebra $S$ and $S$-bicomodule $M$ as follows (It is clear in this case the underlying coring of a comodule, so we omit the subscript $S$ in $\square_S$ as $\square$): 
	$$CoT_{S}(M) = \bigoplus_{n \geq 0}M^{\Box n} = S \oplus M \oplus (M \square M) \oplus (M \square M \square M) \oplus...$$
	Since the ground ring is a field, $M^{\Box n}$ is a vector subspace in $M^{\otimes n}$. Every element in $M^{\Box n}$ could be written as the normal bar notation $[x_1|x_2|...x_n]$ without confusion. This cotensor construction can be endowed with a natural coalgebra structure by deconcatenation: 
	$$\Delta([x_1|...|x_n]) =  [\rho_l(x_1)|...|x_i] + \sum_{i = 1}^{n-1} [x_1|...|x_i] \square [x_{i+1}|...|x_n] + [x_1|x_2|...|\rho_r(x_n)]$$ and counit
	$$\epsilon|_{S} = \epsilon_S \text{~and~} \epsilon|_{M^{\Box n}} = 0 \text{~for~all~}n$$
	Note the cotensor coalgebra has a clear grading by number of tensor factors of $M$. 
	Radford shows that $C(V, G) = CoT_{C[G]}(V)$ exhibiting the following universal property:
	
	\begin{theorem}[Radford, \cite{pointed}] 
		Let $C = k[G(C)] \oplus \overline{C}$ be a pointed coalgebra with a splitting, and $C(S, V)$ constructed from pair $(S, V)$. Let $f$ be a linear map from $\overline{C}$ to $V$. If $f$ is compatible with both bigrading in $\overline{C}$ and $V$, i.e. there is an set function $\phi: G(C) \to S$ s.t. $f({^g\overline{C}{}^h}) \subset ^{\phi(g)}V^{\phi(h)}$. Then there exists a unique coalgebra map $F: C \to C(G,V)$ such that the following diagram of $k$-linear maps commutes
		\[
		\begin{tikzcd}
			\overline{C} \arrow[r, "f"] \arrow[rd, "F|_{\overline{C}}"', dotted] & V                            \\
			& {C(G,V)} \arrow[u, "\pi_V"']
		\end{tikzcd}
		\]
		where $\pi_V$ is the projection from $C(S, V)$ to $V$ and $F({^gC^h}) \subset {^{\phi(g)}C(G,V)^{\phi(h)}}$ for $x, y \in G(C)$.
	\end{theorem}

		$CoT_{C[G]}(V)$ is pointed  with the splitting given by $\delta|_{C[G]} = id$ and $\delta|_{^{\Box n}} = 0$. Now we set $\overline{C(S,V)} = \bigoplus_{n>0}M^{\Box n}$ as the reduced colored coalgebra associated to $C(S,V)$. By the equivalence of reduced colored coalgebras and simply colored colagbras, the universal property becomes 
		\[
		\begin{tikzcd}
			\overline{C} \arrow[r, "f"] \arrow[rd, "\bar{F}"', dotted] & V                            \\
			& {\overline{C(G,V)}} \arrow[u, "\pi_V"']
		\end{tikzcd}
		\]
		Then, it becomes the cofree construction in the category of reduced colored coalgebras, and we will called it \textit{reduced cofree colored coalgebra} over a field.

	\begin{example}
		This theorem generalizes the known fact that \textit{the tensor coalgebra is cofree in the category of (coaugmented) conilpotent coalgebras} by letting $G$ to be a Singleton.  
	\end{example}
	\subsection{Limits and colimits}
Finally, we conclude this section by demonstrating that the category of pointed coalgebras with a splitting (equivalently, the category of reduced colored coalgebras over a field) is both complete and cocomplete. Our approach will differ from that of \cite{ckd}.
	
	\begin{theorem}
		The category of pointed coalgebras with a splitting has all small coproducts and equalizers
	\end{theorem}
	\begin{proof}
Recall that the coproduct of coalgebras corresponds to the direct sum of coalgebras. Given that the sum of pointed coalgebras is also pointed, and the sum of coideals is also a coideal, the coproduct of pointed coalgebras with a splitting inherits the structure from the coproduct in the category of coalgebras. It is precisely the direct sum of pointed coalgebras with a splitting induced by the direct sum of coideals. The argument for the equalizer follows the same reasoning as presented in Theorem 1.1 of \cite{limit}.
	\end{proof}
	
Now we come to the more challenging aspect. Instead of dealing with pointed coalgebras with a splitting, we will shift our focus to reduced colored coalgebras as they offer a less complex framework.

	\begin{lemma} \label{9}
		The category of colored vector spaces has all small products. 
	\end{lemma}
	\begin{proof}
		Given a family $\Lambda$ of colored vector spaces $(\overline{C_\alpha}, G_\alpha)$ for $\alpha \in \lambda$, we define the colored product
		$\mathbb{G}(\prod_{\alpha \in \Lambda}\overline{C_\alpha})$ as the vector subspace of $\prod_{\alpha \in \Lambda}\overline{C_\alpha}$ generated by the sequence of homogeneous elements $({^{g'_{\alpha}}x_\alpha}{}^{g_{\alpha}})_\alpha$ where each ${^{g'_{\alpha}}x_\alpha}{}^{g_{\alpha}} \in {^{g'_{\alpha}}\overline{C_\alpha}}{}^{g_{\alpha}}$
		for some $g'_{\alpha}$ and $g_\alpha$ in $G_{\alpha}$. $\mathbb{G}(\prod_{\alpha \in \Lambda}\overline{C_{\alpha}})$ is $(\prod_{\alpha \in \Lambda}G_{\alpha})$ bigraded by construction. The natural projection from product of vector spaces $p_\alpha$ from $\prod_{\alpha \in \Lambda} \overline{C_\alpha} \to \overline{C_\beta}$ for a $\lambda \in \Lambda$ restricts to the subspace $\mathbb{G}(\prod_{\alpha \in \Lambda}\overline{C_\alpha})$. The restriction $\tilde{p}_\alpha$ of each projection $p_\alpha$ from the universal property on the subspace $\mathbb{G}(\prod_{\alpha \in \Lambda}\overline{C_\alpha})$ is also surjective. The universal property of product of sets induces a map $\pi_\lambda: \prod_{\alpha' \in \Lambda} G_{\alpha'} \to G_{\lambda}$. \newline 
		Now we would like to show $((\mathbb{G}(\prod_{\alpha \in \Lambda}\overline{C_{\alpha}}), \prod_{\alpha \in \Lambda}G_{\alpha}), \{(\tilde{p}_\alpha, \pi_\alpha)\}_{\alpha \in \lambda})$ satisfies the universal property of product.  \newline
		Let $(\overline{D'}, S)$ be a colored vector space and a family of map $(\psi_\alpha, \xi_\alpha): (\overline{D'}, S) \to (\overline{C_\alpha}, G_\alpha)$ for each $\alpha$ in a family $\Lambda$ of a colored vector spaces $(\overline{C_\alpha}, G_\alpha)$. We have $\psi_\alpha ({^g\overline{D'}{}^h}) \subset {^{\xi_{\alpha}(g)}\overline{C_\alpha}{}^{\xi_{\alpha}(h)}}$ and $\xi_{\alpha}(S) \subset G_\alpha$. 
		There are three observations:
		\begin{itemize}
			
			\item The construction of product of sets implies that $\xi_{\alpha}(S) \subset G_\alpha$ for $\alpha \in \lambda$ induces a morphism of sets  $\xi(S) \subset \prod_{\alpha \in \Lambda}G_\alpha$. The map $\xi$ sends $s \in S$ to a sequence of elements $(\xi_{\alpha}(s))_{{\alpha \in \Lambda}}$. 
			\item The construction of product of vector spaces implies that $\psi_\alpha(\overline{D'}) \subset \overline{C_\alpha}$ induces a morphism of vector spaces $\phi: \overline{D'} \to \prod_{\alpha \in \lambda}\overline{C}_{\alpha}$. The map sends $d \in \overline{D'}$ to a sequence of vectors $(\phi_{\alpha}(d))_{{\alpha \in \Lambda}}$
			\item It is an important observation that the ``homogeneous" map $\phi: \overline{D'} \to \prod_{\alpha \in \Lambda} \overline{C_\alpha}$ sends $d \in {^g\overline{D'}{}^h}$ to $(\phi_{\alpha}(d))_{{\alpha \in \Lambda}} \in\prod_{\alpha \in \Lambda} {^{\xi_{\alpha}(g)}  \overline{C_\alpha}{}^{\xi_{\alpha}(h)}}$ because $\psi_\alpha(d) \in {^{\xi_{\alpha}(g)}\overline{C_\alpha}{}^{\xi_{\alpha}(h)}}$. Moreover, ${^{\xi_{\alpha}(g)}\overline{C_\alpha}{}^{\xi_{\alpha}(h)}}$ is in the image of $\tilde{p}_\alpha$.
		\end{itemize}
		Therefore, the image of $\phi$ is in the colored vector spaces $\mathbb{G}(\prod_{\alpha \in \Lambda}\overline{C_{\alpha}})$ and preserves the degrees based on the degree map $\xi$. It follows the family of maps $\{\phi_\alpha, \xi_\alpha\}_{\alpha \in \lambda}$ induces a map $(\phi|_{\mathbb{G}(\prod_{\alpha \in \Lambda}\overline{C_{\alpha}})}, \xi)$ from $(\overline{D'}, S)$ to $(\mathbb{G}(\prod_{\alpha \in \Lambda}\overline{C_{\alpha}}), \prod_{\alpha \in \Lambda}G_\alpha)$. The uniqueness of $(\phi|_{\mathbb{G}(\prod_{\alpha \in \Lambda}\overline{C_{\alpha}})}, \xi)$ follows from the fact that $\xi$ and $\phi$ are unique up to unique isomorphism because of universal properties of products of sets and vector spaces respectively.  
	\end{proof}
	\begin{remark}
The infinite product of vector spaces presents several subtleties. One notable instance is the existence of a non-trivial topology: while many vectors in the product can be expressed as an infinite sum of a set of vectors, not all infinite sums of vectors may belong to this product. This issue is reflected in the fact that $\oplus_i \prod_j M_{i, j} \ncong \prod_j \oplus_iM_{i, j}$, and the usual grading cannot be applied. To address this, we are motivated by the concept of the \textit{graded product} of $\mathbb{N}$-graded modules. By leveraging the property of homogeneous maps, $\oplus_i \prod_j M_{i, j}$ becomes the appropriate choice.
	\end{remark}
	\begin{theorem}
		The category of reduced colored coalgebras over a field has all small coequalizers and small products 
	\end{theorem}
 
	\begin{proof}
		We start with the coeqaulizer first. We would like to coequalizer the morphisms $(f, a), (g, b):(\overline{C}, G) \to (\overline{D}, S)$. It begins with the three observations:
		\begin{itemize}
			\item In the category of sets, there exists the coequalizer (up to unique isomorphism) of morphisms of sets $a, b: G \to S$, which is a quotient set $\frac{S}{\sim}$ and the equivalence relation $\sim$ is generated by $a(x) \sim b(x)$ for all $x \in G$. Let the natural map from $S$ to $\frac{S}{\sim}$ be $p$.
			\item $\overline{D}$ can be endowed with a new bigrading based on $\frac{S}{\sim}$. We set $^{[g]}\overline{D}{}^{[h]}$ equal $\bigoplus\limits_{s, t \in S; p(s) = p(g), p(t) = p(h)} {^{g}\overline{D}{}^{h}}$ for $[g] = p(g), g \in S$. The comultiplication is compatible with the new bigrading: if $\Delta(^g{\overline{D}}{}^h) \subset \sum_{t \in S} ^g{\overline{D}}{}^t \otimes ^t{\overline{D}}{}^h$, then $\Delta(^{[g]}{\overline{D}}{}^{[h]}) \subset \sum_{t \in S} {^{[g]}{\overline{D}}{}^{[t]}} \otimes {^{[t]}{\overline{D}}{}^{[h]}} = \sum_{[t] \in \frac{S}{\sim}} {^{[g]}{\overline{D}}{}^{[t]}} \otimes {^{[t]}{\overline{D}}{}^{[h]}}$ because $^s{\overline{D}}{}^t \subset {^{[s]}{\overline{D}}{}^{[t]}}$ for any $s, t \in S$ by construction.
			\item  $f(x) - g(x) \in {}^{[a(g)]}\overline{D}{}^{[a(h)]}$ for $x \in {}^g\overline{C}{}^h$ and $p(a(g)) = p(b(g)) = [a(g)]$ ($p(a(h)) = p(b(h)) = [a(h)]$). It implies that 
			$$\im(f - g) = \bigoplus_{[s], [t] \in \frac{S}{\sim}}({}^{[s]}\overline{D}{}^{[t]} \cap \im(f - g))$$ i.e. it is generated by homogeneous elements. $f$ and $g$ that restrict to ${}^g\overline{C}{}^h$ is coequalized in $\frac{{}^{[a(g)]}\overline{D}{}^{[a(h)]}}{{}^{[a(g)]}\overline{D}{}^{[a(h)]} \cap \im(f - g)}$. In addition, $\im(f - g)$ is a graded coideal $\bigoplus_{\text{cond}_2}{{}^{s'}\overline{C}{}^{t'}}$. The check is the same as in the case of normal coalgebra, but we do not need to check counit condition.

		\end{itemize}    
		We claim $(\bigoplus_{[s], [t] \in \frac{S}{\sim}}\frac{{}^{[s]}\overline{D}{}^{[t]}}{{}^{[s]}\overline{D}{}^{[t]} \cap \im(f - g)}, (\pi, p))$ is the coequalizer (up to unique isomorphism) of $(f, a)$ and $(g, b)$ where $\pi$ is the natural projection which is homogeneous according to $p$. Suppose there is another $((\overline{B}, Q), (h, c))$ that coequalizes the two morphisms. On the level of set, there exists a unique morphism $\tilde{c}$ from $\frac{S}{\sim} \to Q$ s.t. $\tilde{c} \circ p = c$.  
		\[
		\begin{tikzcd}
			G \arrow[r, "a", shift left] \arrow[r, "b"', shift right] & S \arrow[r, "p"] \arrow[rd, "c"] & \frac{S}{\sim} \arrow[d, "\exists!\tilde{c}", dashed] \\
			&                                  & Q                                                    
		\end{tikzcd}
		\]
		This induces the following diagram of homogeneous maps on the level of vector spaces. Fixing two element $[s], [t]$ in $\frac{S}{\sim}$. Name \textbf{condition 1} to be $p(a(m)) = p(b(m)) = [s]$, $p(a(n)) = p(b(n)) = [t]$ and $m, n \in G$. Likewise, name the \textbf{condition 2} to be $p(s') = [s]$, $p(t') = [t]$ and $s', t' \in S$. The goal is to simplify the writing under the plus sign. There exists a coequalizer on the level of vector spaces. 
		
		\[  
		\begin{tikzcd}
			\bigoplus_{\text{cond}_1}{{}^{m}\overline{C}{}^{n}} \arrow[r, "{(f, a)_\text{res}}", shift left] \arrow[r, "{(g, b)_\text{res}}"', shift right] & \bigoplus_{\text{cond}_2}{{}^{s'}\overline{D}{}^{t'}} \arrow[r, "\tilde{\pi}"] \arrow[rd, "{(h, c)}"'] & {\frac{{}^{[s]}\overline{D}{}^{[t]}}{\im(f - g)_\text{res}}} \arrow[d, "\tilde{h}_{[s], [t]}", dashed]\\
			&                                                                                                                  & {^{c(s)}\overline{B}{}^{c(t)}}                                                      
		\end{tikzcd}
		\]
		Note $\bigoplus_{\text{cond}_2}{{}^{s'}\overline{D}{}^{t'}} = {{}^{[s]}\overline{D}{}^{[t]}}$ as vector spaces (but not colored ones). Thus the natural projection $\tilde{\pi}$ is in fact the restriction of $(\pi, p)$. The dash arrow is unique due to isomorphism theorems of vector spaces; it is defined as $\tilde{h}_{[s], [t]}(\tilde{\pi}(s)) = h(s)$ (remember $\tilde{\pi}$ is surjective). $\im(f - g)_\text{res}$ is the image of $f - g$ that restricts on $\bigoplus_{\text{cond}_1}{{}^{m}\overline{C}{}^{n}}$. Because of previous observations, $\im(f - g)_\text{res}$ is Furthermore, 
		$$\frac{{}^{[s]}\overline{D}{}^{[t]}}{{}^{[s]}\overline{D}{}^{[t]} \cap \im(f - g)}  = \frac{{}^{[s]}\overline{D}{}^{[t]}}{\im(f - g)_\text{res}}$$
		Then glue $\tilde{h}_{[s], [t]}$ for each element $[s], [t] \in \frac{S}{\sim}$; it yields $(\tilde{h}, \tilde{c}) = \bigoplus_{[s], [t] \in \frac{S}{\sim}}\tilde{h}_{[s], [t]}$ as a morphism between colored vector spaces. Since $\im(f - g)$ is an coideal, this morphism is furthermore a morphism between reduced colored coalgebras, which is unique because $\tilde{c}$ is unique on the level of sets and $\tilde{c}$ is unqieu on the level of vector spaces (and respect the bigrading). It shows  $(\bigoplus_{[s], [t] \in \frac{S}{\sim}}\frac{{}^{[s]}\overline{D}{}^{[t]}}{{}^{[s]}\overline{D}{}^{[t]} \cap \im(f - g)}, (\pi, p))$ is the coequalizer between $(f, a)$ and $(g, b)$. \newline

		Next, the idea of constructing the product is the same as the one in theorem 1.1 in \cite{limit}, but there are some modification to the argument: Let $(\overline{C_i}, G_i)_{i \in I}$ be a a family of reduced colored coalgebras and (Lemma \ref{9}) the associated reduced cofree colored coalgebra $\overline{C(\mathbb{G}(\prod_{i \in I}\overline{C_{i}}), \prod_{i\in I}G_i)}$ over the colored vector spaces $\prod_{i \in I}\overline{C_{i}}$ where we forget the coalgebra structure of  $\overline{C_{i}}$. The rest follows \cite{limit}. Let $(D, \prod_{i\in I}G_i)$ be the sum of all subcoalgebras $(E, \prod_{i\in I}G_i)$ of $\overline{C(\mathbb{G}(\prod_{i \in I}\overline{C_{i}}), \prod_{i\in I}G_i)}$ such that $(\tilde{p_i}, \pi_i) \circ (\pi_{\mathbb{G}(\prod_{i \in I}\overline{C_{i}})}, id) \circ (j_E, id)$ where $(j_E, id)$ is the canonical inclusion from $E$ to $\mathbb{G}(\prod_{i \in I}\overline{C_{i}})$. Then one can complete the proof by following the same method in \cite{limit} to check $$((D, \prod_{i\in I}G_i), ((\tilde{p_i}, \pi_i) \circ (\pi_{\mathbb{G}(\prod_{i \in I}\overline{C_{i}})}, id) \circ (j_E, id))_{i \in I})$$ satisfies the universal property and thus the product of the family $I$ in the category of reduced colored coalgebras. 
		
	\end{proof}
	
	Lastly, 
	\begin{corollary}
		The category of pointed coalgebras with a splitting (or equivalently, the category of reduced colored coalgebras) has all small limits and colimits
	\end{corollary}
	\begin{remarks}
		Our argument also works if the underlying $R$-module is either $\mathbb{N}$-graded or differential $\mathbb{N}$-graded. First, we need to endow $C[G]$ with a trivial graded structure that it lies only in zero degree and trivial differentials. Second, every morphisms considered in the paper should be also homogeneous and commutes with differentials. 
	\end{remarks}

	\bibliography{mainbib}
	\bibliographystyle{amsalpha}
	
\end{document}